\documentclass[12pt]{amsart}

\usepackage{amssymb,amsthm,amsmath}
\usepackage{fullpage}
\usepackage{paralist}
\usepackage{graphicx}
\usepackage{hyperref}

\theoremstyle{plain}
\newtheorem{thm}{Theorem}[section]
\newtheorem{lem}[thm]{Lemma}
\newtheorem{prop}[thm]{Proposition}
\newtheorem{cor}[thm]{Corollary}
\newtheorem*{conj}{Conjecture}

\theoremstyle{definition}
\newtheorem{algm}[thm]{Algorithm}
\newtheorem{dfn}[thm]{Definition}
\newtheorem{ex}[thm]{Example}

\theoremstyle{remark}
\newtheorem{rmk}[thm]{Remark}

\title{Rank-determining sets of metric graphs}
\author{Ye Luo}
\address{School of Mathematics, Georgia Institute of Technology, Atlanta, GA 30332}
\email{luoye@math.gatech.edu}


\date{}
\begin{document}
\maketitle

\begin{abstract}
\noindent
A metric graph is a geometric realization of a finite graph by identifying each edge with a real interval. A divisor on a metric graph $\Gamma$ is an element of the free abelian group on $\Gamma$. The rank of a divisor on a metric graph is a concept appearing in the Riemann-Roch theorem for metric graphs (or tropical
curves) due to Gathmann and Kerber~\cite{GK07}, and Mikhalkin and Zharkov~\cite{MZ07}. We define a \emph{rank-determining set} of a metric graph $\Gamma$ to be a
subset $A$ of $\Gamma$ such that the rank of a divisor $D$ on $\Gamma$ is always equal to the rank of $D$ restricted
on $A$. We show constructively in this paper that there exist finite
rank-determining sets. In addition, we investigate the properties of
rank-determining sets in general and formulate a criterion for
rank-determining sets. Our analysis is a based on an algorithm to
derive the $v_0$-reduced divisor from any effective divisor in the
same linear system.
\end{abstract}

\section{Introduction}
In the past few years, people have been attracted to investigate the
analogies and connections among linear systems on algebraic curves,
finite graphs, metric graphs and tropical curves
\cite{B07,BN07,GK07,HKN07,MZ07}. In particular, a recent work of
Hladk\'{y}, Kr\'{a}l' and Norine \cite{HKN07} shows that the rank of
a divisor $D$ on a graph equals the rank of $D$ on the corresponding
metric graph $\Gamma$. However, their result requires that all the
edges of $\Gamma$ have length $1$ and $D$ is zero on the interiors
of the edges. As an initial step of this paper, we assert that these
restrictions are not necessary by proving that for an arbitrary
metric graph $\Gamma$ with a vertex set $\Omega$ and an arbitrary
divisor $D$ on $\Gamma$, the rank $r(D)$ of $D$ equals the $\Omega$-restricted rank $r_\Omega(D)$ of $D$. This
result motivates us into further investigations on the subsets of
$\Gamma$ having such a property, to which we give the name
\emph{rank-determining sets}.

\subsection{Preliminaries}
Throughout this paper, a \emph{graph} $G$ means a finite connected
multigraph with no loop edges, and a \emph{metric graph} $\Gamma$
means a graph having each edge assigned a positive length. And
roughly speaking, a \emph{tropical curve} is a metric graph where we
admit some edges incident with vertices of degree $1$ having infinite
length \cite{M06}\cite{MZ07}. We will expand our discussions within
the framework of metric graphs, while the conclusions also apply for
tropical curves.

Denote the vertex set and the edge set of a graph $G$ by $V(G)$ and
$E(G)$, respectively. The \emph{genus} $g$ of $G$ is the first Betti
number of $G$ or the maximum number of independent cycles of $G$,
which equals $\#E(G)-\#V(G)+1$.

We can also define vertices and edges on a metric graph $\Gamma$. We
call $\Omega$ a \emph{vertex set} of $\Gamma$ and the elements of
$\Omega$ \emph{vertices}, if $\Omega$ is a nonempty finite subset of
$\Gamma$ satisfying the following conditions:
\begin{compactenum}[(i)]
\item $\Gamma\setminus\Omega$ is a disjoint union of subspaces $e^o_i$ isometric to open intervals.
\item Let $e_i$ be the closure of $e^o_i$. For all $i$, $e_i\setminus e^o_i$ contains exactly two distinct points, which are both elements of
$\Omega$. We call $e_i$ an \emph{edge} of $\Gamma$, $e^o_i$ the
\emph{interior} of $e_i$, and $v\in e^o_i$ an \emph{internal point} of
$e_i$. And we say that the two vertices in $e_i\setminus e^o_i$ are
two \emph{ends} (or \emph{end-points}) of $e_i$ or $e^o_i$, while $e_i$ is an edge
\emph{connecting} these vertices.
\end{compactenum}
Clearly, $\Gamma$ is loopless with respect to $\Omega$. And by our
definition of a vertex set, there might be multiple edges between
two vertices, which is not allowed in definitions of vertex sets by
other authors (see, e.g., \cite{BF06}). Throughout this paper,
whenever we mention a vertex or an edge of a metric graph $\Gamma$,
we always assume a vertex set of $\Gamma$ is predetermined, whether
or not it is presented explicitly. Given a vertex set of $\Gamma$,
the genus of $\Gamma$ can be computed just like in the graph case
(note that the genus is independent of how we choose vertex sets).

By identifying each edge with a closed interval, the subintervals
are called \emph{segments} of $\Gamma$. The boundary points of a
segment are called the \emph{ends} (or \emph{end-points}) of that segment. In
addition, we transport the conventional notations for intervals onto
metric graphs. For example, let $w_1$ and $w_2$ be two vertices that
are neighbors, $e$ be one of the edges connecting them, and $v$ be
an internal point $e$. Then $(w_1,w_2)$ represents all the internal
points of the edges connecting $w_1$ and $w_2$. And to avoid
confusion in case of multiple edges, $e$ can be represented by
$[w_1,v,w_2]$. We use $\text{dist}(x,y)$ to denote the distance
between two points $x$ and $y$ measured on $\Gamma$, and define the
distance between two subsets $X$ and $Y$ of $\Gamma$, denoted by
$\text{dist}(X,Y)$, to be $\inf\{\text{dist}(x,y), x\in X, y\in
Y\}$. If $e'$ is a segment, and $x,y\in e'$, then we use
$\text{dist}_{e'}(x,y)$ to denote the
distance between $x$ and $y$ measured on $e'$. 

For simplicity of notation, if $v$ is a point
of a metric graph, sometimes we refer to the singleton $\{v\}$ by just
writing $v$.

A divisor $D$ on $G$ is an element of the free abelian group
$\text{Div}G$ on the vertex set of $G$. We can uniquely write a
divisor $D\in \text{Div}G$ as $D=\sum_{v\in V(G)}D(v)(v)$, where
$D(v)\in\mathbb{Z}$ evaluates $D$ at $v$. The \emph{degree} of $D$
is defined by the formula $\text{deg}(D)=\sum_{v\in V(G)}D(v)$. A
divisor $D$ is called \emph{effective} if $D(v)\geqslant 0$ for all
$v\in V(G)$. We denote the set of all effective divisors on $G$ by
$\text{Div}_+G$, and the set of all effective divisors of degree $s$
on $G$ by $\text{Div}_+^sG$. Provided a function $f:V(G)\rightarrow
\mathbb{Z}$, the divisor associated to $f$ is given by
\begin{equation*}
D_f=\sum_{v\in V(G)}\sum_{e=wv\in E(G)}(f(v)-f(w))(v),
\end{equation*}
and called \emph{principal}. It is easy to see that the principal
divisors have degree $0$. For two divisors $D$ and $D'$, we say that
$D$ is \emph{linearly equivalent} to $D'$ or $D\sim D'$ if $D-D'$
is principal. And we defined the \emph{linear system associated to a
divisor} $D$ to be the set $|D|$ of all effective divisors linearly
equivalent to $D$. Since $|D|$ does not have a pure dimension, Baker
and Norine \cite{BN07} introduced the concept of the \emph{rank} of a
divisor $D$, denoted by $r_G(D)$, to describe the dimensional aspect
of $|D|$ . Explicitly, $r_G(D)=-1$ if $|D|=\emptyset$, and
$r_G(D)\geqslant s\geqslant 0$ if and only if $|D-E|\neq \emptyset$ for all
$E\in\text{Div}_+^sG$. When it is clear that $D$ is defined on $G$,
we usually omit the subscript and write $r(D)$ instead of $r_G(D)$.

Analogously, for a metric graph (or a tropical curve) $\Gamma$,
elements of the free abelian group $\text{Div}\Gamma$ on $\Gamma$
are called divisors on $\Gamma$. We can define the degree of a
divisor and the notion of effective divisors in a similar way. A
rational function $f$ on $\Gamma$ is a continuous, piecewise linear
real function with integral slopes. The \emph{order} $\text{ord}_vf$
of $f$ at a point $v\in\Gamma$ is the sum of the outgoing slopes of
all the segments emanating from $v$. Any rational function $f$ has an associated a divisor
$(f):=\sum_{v\in\Gamma}\text{ord}_vf\cdot(v)$. We say $(f)$ is
\emph{principal} for all rational functions $f$, and define linear
equivalence relations and linear systems as on graphs. Also, we may
define the \emph{rank} $r_\Gamma(D)$ of a divisor $D$ on $\Gamma$.
Explicitly, $r_\Gamma(D)=-1$ if $|D|=\emptyset$, and
$r_\Gamma(D)\geqslant s\geqslant 0$ if and only if $|D-E|\neq \emptyset$ for
all $E\in\text{Div}_+^s\Gamma$. We may omit the subscript and use
$r(D)$ to represent the rank of a divisor $D$, when there is no
confusion that $D$ is defined on $\Gamma$.

\subsection{Overview}
As an analogue of the classical Riemann-Roch theorem on Riemann
surfaces, Baker and Norine formulated and proved the Riemann-Roch
theorem for the rank of divisors on finite graphs \cite{BN07}. We define the \emph{canonical divisor} on a graph $G$ to be the divisor $K$ given by
$K=\sum_{v\in V(G)}(\text{deg}(v)-2)(v)$.

\begin{thm}[Riemann-Roch thoerem for graphs] \label{T:RR-graph}
Let $G$ be a graph of genus $g$ and $K$ the canonical divisor on $G$. Then for all $D\in\emph{Div}G$, we have
\begin{equation*}
r_G(D)-r_G(K-D)=\emph{deg}(D)+1-g.
\end{equation*}
\end{thm}

Not long after, such an analogy was extended to metric graphs and
tropical curves by Gathmann and Kerber \cite{GK07}, by Hladk\'{y},
Kr\'{a}l' and Norine \cite{HKN07}, and by Mikhalkin and Zharkov
\cite{MZ07}. For a metric graph (or a tropical curve) $\Gamma$, we may also define the \emph{canonical divisor} on $\Gamma$ to be the divisor $K$ given by $K=\sum_{v\in \Gamma}(\text{deg}(v)-2)(v)$.

\begin{thm}[Riemann-Roch thoerem for metric graphs and tropical curves] \label{T:RR-metricgraph}
Let $\Gamma$ be a metric graph (or a tropical curve) of genus $g$ and $K$ the canonical divisor on $\Gamma$. Then for all $D\in\emph{Div}\Gamma$, we have
\begin{equation*}
r_\Gamma(D)-r_\Gamma(K-D)=\emph{deg}(D)+1-g.
\end{equation*}
\end{thm}

The following theorem, conjectured by Baker and proved by
Hladk\'{y}, Kr\'{a}l' and Norine \cite{HKN07}, states another
important property about rank of divisors. For a graph $G$, by
assigning all edges length $1$, we obtain a metric graph
\emph{corresponding to} $G$.

\begin{thm} \label{T:graph}
Let $\Gamma$ be the metric graph corresponding to a graph $G$. Let
$D$ be a divisor on $G$. Let $r_G(D)$ be the rank of $D$ on $G$, and
$r_\Gamma(D)$ the rank of $D$ on $\Gamma$. Then we have
$r_G(D)=r_\Gamma(D)$.
\end{thm}

We introduce a new notion of rank here.
\begin{dfn}
Let $\Gamma$ be a metric graph and $A$ a nonempty subset of
$\Gamma$.
\begin{compactenum}[(i)]
\item Define the \emph{$A$-restricted rank} $r_A(D)$ of a divisor $D\in\text{Div}\Gamma$ by $r_A(D)=-1$ if $|D|=\emptyset$, and $r_A(D)\geqslant
s\geqslant 0$ if and only if $|D-E|\neq \emptyset$ for all
$E\in\text{Div}_+^sA$.
\item $A$ is said to be a \emph{rank-determining
set} of $\Gamma$, if it holds for every divisor $D\in\text{Div}\Gamma$ that
$r(D)=r_A(D)$.
\end{compactenum}
\end{dfn}

One may also call $r_A(D)$ the rank of $D$ restricted on $A$. Clearly, $\Gamma$ itself is a rank-determining set of $\Gamma$ and
we say it is \emph{trivial}. It is natural to ask if there exist
nontrivial rank-determining sets, or more ambitiously, finite ones?
One of the main results of this paper is the following theorem,
which gives an affirmative answer.

\begin{thm} \label{T:vertex}
Let $\Omega$ be a vertex set of a metric graph $\Gamma$. Then
$\Omega$ is a rank-determining set of $\Gamma$.
\end{thm}

It is easy to see that Theorem~\ref{T:vertex} generalizes Theorem~\ref{T:graph} to all metric graphs $\Gamma$ and all divisors $D$ on
$\Gamma$. And since $\text{Div}_+^s\Omega$ is always a finite set,
this theorem also provides an algorithm for computing the rank of a
divisor on $\Gamma$.

There exist finite rank-determining sets other than vertex
sets. In particular, we will prove the following conjecture of
Baker.

\begin{thm} \label{T:g+1}
Let $\Gamma$ be a metric graph of genus $g$. Then there exists a
finite rank-determining set of cardinality $g+1$.
\end{thm}

Theorem~\ref{T:g+1} has a counterpart in the algebraic curve case,
as stated in the following theorem. (See
Remark~\ref{R:ac} for a sketch of the proof.)

\begin{thm}[R. Varley] \label{T:ac}
For a nonsingular projective algebraic curve $C$, any set of $g+1$
distinct points is a rank-determining set.
\end{thm}

It is clear that the equivalence relation among divisors on
$\Gamma$ changes if we use a different metric. However,
rank-determining sets will not be affected, even though their
definition uses the notion of linear systems on $\Gamma$.

\begin{thm} \label{T:homeo}
Rank-determining sets are preserved under homeomorphisms.
\end{thm}

In Section~\ref{S:1}, we present an algorithm for computing the
$v_0$-reduced divisor linearly equivalent to a given effective
divisor on $\Gamma$. In Section~\ref{S:2}, we investigate properties
of rank-determining sets based on this algorithm, which are
generalized into a subtle criterion for rank-determining sets, from
which Theorem~\ref{T:vertex}, \ref{T:g+1} and \ref{T:homeo} easily
follow. We also explore several concrete examples as applications of
the criterion.

\bigskip
{\sc Acknowledgments:}
I would most of all like to thank Matthew Baker for introducing me to
this topic, for his encouragement of further study on general
properties of rank-determining sets, and for many valuable
discussions. Dr. Baker also helped me simplify the proof of Theorem~\ref{T:tmt} and gave detailed comments on the draft. Thanks to Robert Varley for providing his proof of Theorem~\ref{T:ac} in
Remark~\ref{R:ac}. I would also like to thank Serguei Norine for helpful
discussions, and Josephine Yu, Robin Thomas, Prasad Tetali and Farbod Shokrieh for their comments.

\section{From effective divisors to reduced ones} \label{S:1}
\subsection{Reduced divisors}
The notion of \emph{reduced divisors} was adopted
in \cite{BN07} as an important tool in the proof of the Riemann-Roch
theorem for finite graphs. The definition of reduced divisors on finite graphs is based on
the notion of \emph{$G$-parking functions} \cite{PS04}.

Let $G$ be a finite graph. For $A\subseteq V(G)$ and $v\in A$, the
\emph{out-degree} of $v$ from $A$, denoted by $\text{outdeg}_A(v)$,
is defined as the number of edges of $G$ with one end at $v$ and the
other end in $V(G)\setminus A$. Choose a vertex $v_0$. We say a
function $f:V(G)\setminus \{v_0\}\rightarrow \mathbb{Z}$ is a
\emph{$G$-parking function} based at $v_0$ if
\begin{compactenum}[(i)]
\item $f(v)\geqslant 0$ for all $v\in V(G)\setminus \{v_0\}$, and
\item every nonempty subset $A$ of $V(G)\setminus \{v_0\}$ contains a vertex
$v$ such that $f(v)<\text{outdeg}_A(v)$.
\end{compactenum}

A divisor $D\in\text{Div}(G)$ is called \emph{$v_0$-reduced} if the
map $v\mapsto D(v)$ restricted on $V(G)\setminus \{v_0\}$ is a
$G$-parking function based at $v_0$. An important property of
reduced divisors is stated in the following proposition.

\begin{prop}[See Proposition~3.1 in \cite{BN07}] \label{P:1:1}
If we fix a base vertex $v_0\in V(G)$, then for every
$D\in\emph{Div}G$, there exists a unique $v_0$-reduced divisor
$D'\in\emph{Div}G$ such that $D'\sim D$.
\end{prop}

Proposition~\ref{P:1:1} is quite useful when dealing with
equivalence classes of divisors, since we can select a reduced
divisor as a concrete representative for each equivalence class of
divisors.

The notion of reduced divisors has been extended to metric graphs by
several authors. In this paper, we adopt the definition of reduced
divisors on metric graphs as in \cite{HKN07}, which follows closely
the definition of reduced divisors on finite graphs as discussed
above. Other authors suggest to define reduced divisors on metric
graphs in more abstract ways \cite{B08}\cite{MZ07}, and it can be
proved that these definitions are all equivalent.

Let $\Gamma$ be a metric graph. If $X$ is a subset of $\Gamma$ with finitely many connected
components, we use $X^c$ to denote the complement of $X$ on
$\Gamma$, $\overline{X}$ the closure of $X$, $X^o$ the interior of
$X$, and $\partial X$ the set of boundary points of $X$. Note that
$\partial X=\partial X^c$. In addition, if $X$ is closed, then for
$v\in\partial X$, we define the \emph{out-degree} of $v$ from $X$,
denoted by $\text{outdeg}_X(v)$, to be the number of edges leaving
$X$ at $v$, or more precisely, the maximum number of internally
disjoint segments of $X^c$ with an open end at $v$. For $D\in
\text{Div}\Gamma$, we call a boundary point $v$ of $X$
\emph{saturated} with respect to $X$ and $D$  if $D(v)\geqslant
\text{outdeg}_X(v)$, and \emph{non-saturated} otherwise.

\begin{dfn}
Fix a base point $v_0\in \Gamma$. We say that a divisor $D$ is
$v_0$-\emph{reduced} if $D$ is non-negative on $\Gamma\setminus
v_0$, and every closed connected subset $X$ of $\Gamma\setminus v_0$
contains a non-saturated point $v\in
\partial X$.
\end{dfn}

As a counterpart of Proposition~\ref{P:1:1}, the following theorem
asserts the existence and uniqueness of a $v_0$-reduced divisor in
any equivalence class of $\text{Div}\Gamma$ \cite{HKN07}\cite{MZ07}.

\begin{thm} \label{T:unique}
Let $D$ be a divisor on a metric graph $\Gamma$. For any $v_0\in
\Gamma$, there exists a unique $v_0$-reduced divisor $D_{v_0}$ that
is linearly equivalent to $D$.
\end{thm}

For any finite subset $S$ of $\Gamma$, we denote by
$\mathcal{U}_{S,v_0}$ the maximal connected subset of $\Gamma$, such
that $v_0\in \mathcal{U}_{S,v_0}$ and $S\bigcap
\mathcal{U}_{S,v_0}=\emptyset$. In particular, if $v_0\in S$, then
$\mathcal{U}_{S,v_0}=\emptyset$. The set $\mathcal{U}_{S,v_0}$ can
be derived by taking the connected component of $S^c$ which contains
$v_0$. Note that $\mathcal{U}_{S,v_0}$ is connected and open, while
$\mathcal{U}_{S,v_0}^c$ is closed and might have several connected
components. We say that $S$ is \emph{$v_0$-minimal} if
$\mathcal{U}_{S,v_0}^c$ is connected and $S$ equals the set of
boundary points of $\mathcal{U}_{S,v_0}^c$.

Let $D$ be a divisor on $\Gamma$. Let $\textrm{supp}D=\{v\in\Gamma|D(v)\neq 0\}$ and $\textrm{supp}|D|=\bigcup_{D'\in |D|}\textrm{supp}D'$.
We call $\textrm{supp}D$ the \emph{support of $D$} and call $\textrm{supp}|D|$ the \emph{support of $|D|$}.

Assume now that $D$ is effective. To verify if $D$ is $v_0$-reduced,
we do not need to go through all closed connected subsets of
$\Gamma\setminus v_0$. The following lemma shows that we only need
to consider finitely many of them.

\begin{lem}\label{L:1:1}
Let $v_0$ be a point of $\Gamma$ and $D$ an effective divisor on
$\Gamma$. Then $D$ is $v_0$-reduced if and only if for any subset $S$ of
$\emph{supp}D \setminus v_0$, $\mathcal{U}_{S,v_0}^c$ contains a
non-saturated boundary point with respect to $D$.
\end{lem}

\begin{proof}
First assume $D$ is $v_0$-reduced and consider a subset $S$ of
$\textrm{supp}D \setminus v_0$. Then $\mathcal{U}_{S,v_0}^c$ is a
closed subset of $\Gamma$ which has finitely many components. Apply
the defining property of $v_0$-reduced divisors to any of these
components, and we obtain non-saturated boundary points on each of
them.

Conversely, assume that for any subset $S$ of $\textrm{supp}D
\setminus v_0$, $\mathcal{U}_{S,v_0}^c$ contains a non-saturated
point. If $D$ is not $v_0$-reduced, then there exists a closed
connected subset $X$ of $\Gamma \setminus v_0$, such that every
point of $\partial X$ is saturated with respect to $X$ and $D$.
Clearly $\partial X\subseteq \textrm{supp}D \setminus v_0$. And
since $X\subseteq \mathcal{U}_{\partial X,v_0}^c$, the edges leaving
$\mathcal{U}_{\partial X,v_0}^c$ must also be edges leaving $X$.
Therefore, for every $v\in \partial \mathcal{U}_{\partial X,v_0}$,
we have
\begin{equation*}
D(v)\geqslant \textrm{outdeg}_X(v)\geqslant
\textrm{outdeg}_{\mathcal{U}_{\partial X,v_0}^c}(v).
\end{equation*}
This is equivalent to saying that $\mathcal{U}_{\partial X,v_0}^c$
contains no non-saturated boundary points, which contradicts our
assumption.
\end{proof}

Lemma~\ref{L:1:1} tells us that to determine if an effetive divisor
$D$ is $v_0$-reduced, it suffices to consider only the subsets of
$\text{supp}D \setminus v_0$. But the number of cases still grows
exponentially with respect to $\#\text{supp}D$. For finite graphs,
there is an elegant algorithm for verifying if a given function is a
$G$-parking function, which is adapted from an algorithm provided
by Dhar \cite{D90} in the context of sandpile models (see
\cite{CP05}). Here we extend Dhar's algorithm to metric graphs, as a consequence of
which we just need to test the points in $\text{supp}D \setminus
v_0$ one by one in order to judge whether an effective divisor $D$
is $v_0$-reduced.

\begin{algm}\label{A:Dhar}
(\textbf{Dhar's algorithm for metric graphs})\\
\textbf{Input:} An effective divisor $D\in \text{Div}_+\Gamma$, and a point $v_0\in\Gamma$.\\
\textbf{Output:} A subset $S$ of  $\text{supp}D\setminus v_0$.\\
Initially, set $S_0=\textrm{supp}D \setminus v_0$, and $k=0$.
\begin{compactenum}[(1)]
\item If $S_k=\emptyset$ or all the boundary points of $\mathcal{U}_{S_k,v_0}^c$
are saturated with respect to $D$, set $S=S_k$ and stop the
procedure.
\item Let $N_k$ be the set of all non-saturated boundary points of
$\mathcal{U}_{S_k,v_0}^c$. Set $S_{k+1}= S_k\setminus N_k$. Set
$k\leftarrow k+1$ and go to step (1).
\end{compactenum}
\end{algm}

\begin{lem} \label{L:1:2}
Run Dhar's algorithm for an effective divisor $D$ and a point $v_0$.
Then $D$ is $v_0$-reduced if and only if the output $S$ is empty.
\end{lem}
\begin{proof}
If $S$ is nonempty, then all the boundary points of
$\mathcal{U}_{S,v_0}^c$ are saturated. Thus $D$ is not $v_0$-reduced
by Lemma~\ref{L:1:1}.

Otherwise, $S=\emptyset$. For a subset $S'$ of $\text{supp}D
\setminus v_0$, let $N_k$ be such that $N_k\bigcap
S'\neq \emptyset$ and $N_{k'}\bigcap
S'= \emptyset$ for $k'<k$. Note that $S'\subseteq S_k$. If $v\in N_k\bigcap
S'$, then $v$ must be a non-saturated boundary point of
$\mathcal{U}_{S',v_0}^c$, since
\begin{equation*}
D(v)< \textrm{outdeg}_{\mathcal{U}_{S_k,v_0}^c}(v)\leqslant
\textrm{outdeg}_{\mathcal{U}_{S',v_0}^c}(v).
\end{equation*}
By Lemma~\ref{L:1:1}, $D$ is $v_0$-reduced.
\end{proof}

\begin{ex} \label{E:Dhar}
\begin{figure}[h]
\centering
\includegraphics[width=0.75\textwidth]{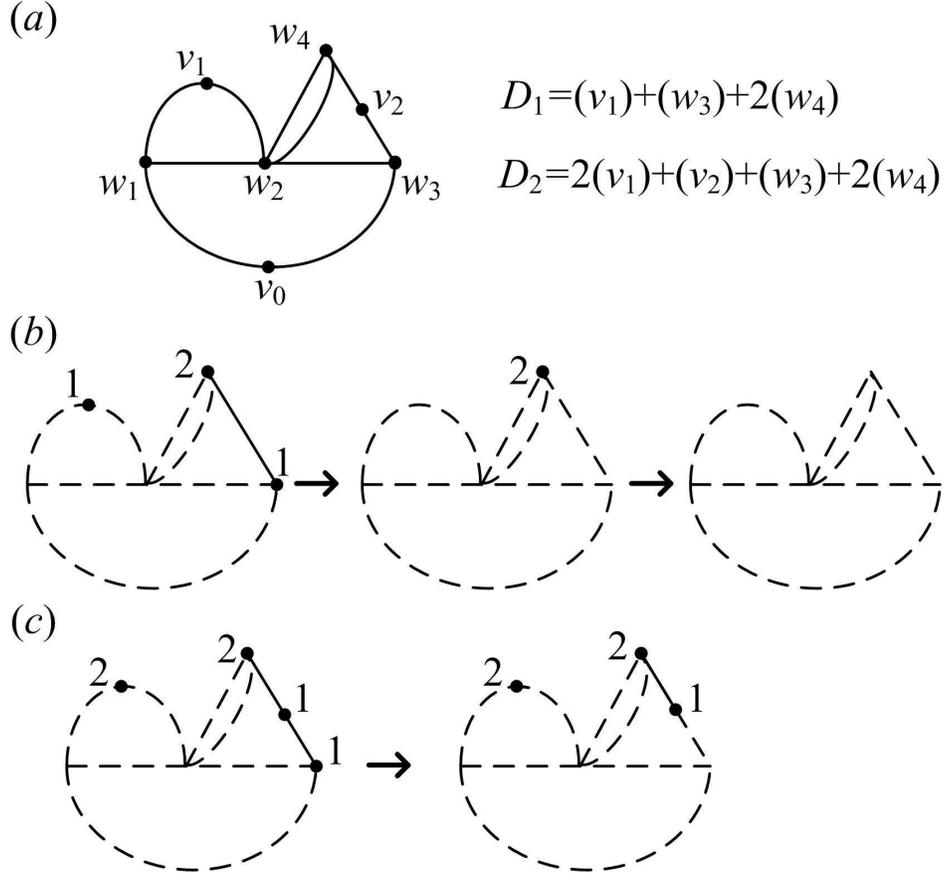}
\caption{(a) A metric graph $\Gamma$ and two effective divisors
$D_1$ and $D_2$ on $\Gamma$. (b) Dhar's algorithm for $D_1$ and
$v_0$. (c) Dhar's algorithm for $D_2$ and $v_0$.} \label{F:Dhar}
\end{figure}
Let $\Gamma$ be a metric graph as illustrated in Figure~\ref{F:Dhar}(a) with a vertex set $\{w_1,w_2,w_3,w_4\}$. Let
$D_1=(v_1)+(w_3)+2(w_4)$ and $D_2=2(v_1)+(v_2)+(w_3)+2(w_4)$. Run
Dhar's algorithm for $D_1$ and $v_0$. The dashed areas in Figure~\ref{F:Dhar}(b) illustrate $\mathcal{U}_{S_k,v_0}$ step by step.
Initially, we have $S_0=\{v_1,w_3,w_4\}$ and
$\mathcal{U}_{S_0,v_0}^c=\{v_1\}\bigcup[w_3,w_4]$. The set $N_0$ of all
non-saturated boundary points of $\mathcal{U}_{S_0,v_0}^c$ is
$\{v_1,w_3\}$. Then $S_1=S_0\setminus N_0=\{w_4\}$ and
$\mathcal{U}_{S_1,v_0}^c=\{w_4\}$. Since $w_4$ is a non-saturated point,
we have $N_1=\{w_4\}$ and $S_2=\emptyset$. Now $\mathcal{U}_{S_2,v_0}^c$
is the whole graph and we get the output $S=\emptyset$. Therefore
$D_1$ is $v_0$-reduced. We leave it to the readers to verify the
output of Dhar's algorithm for $D_2$ and $v_0$ is $\{v_1,v_2,w_4\}$
and $D_2$ is not $v_0$-reduced (Figure~\ref{F:Dhar}(c)).
\end{ex}

\begin{rmk}
The out-degrees are topological invariants, which implies that whether or not a
divisor is $v_0$-reduced is preserved under homeomorphisms.
\end{rmk}

\subsection{An algorithm for computing reduced divisors}

Based on Dhar's algorithm and the criterion from Lemma~\ref{L:1:2},
we formulate an algorithm to derive from an effective divisor $D$
the unique $v_0$-reduced divisor linearly equivalent to $D$.

Recall from \cite{HKN07} the notion of \emph{basic $v_0$-extremal
functions} on $\Gamma$. We say a rational function $f$ is a
\emph{basic $v_0$-extremal function} if there exist closed connected
disjoint subsets $X_{\text{max}}(f)$ and $X_{\text{min}}(f)$ of
$\Gamma$ such that:
\begin{compactenum}[(i)]
\item $v_0\in X_{\text{min}}(f)$;\\
\item $\Gamma-X_{\text{max}}(f)-X_{\text{min}}(f)$ is the union
of disjoint open segments of the same length; \\
\item $f$ achieves its maximum on $X_{\text{max}}(f)$ and its
minimum on $X_{\text{min}}(f)$;\\
\item $f$ has constant slope $1$ from $X_{\text{min}}(f)$ to
$X_{\text{max}}(f)$ on $\Gamma-X_{\text{max}}(f)-X_{\text{min}}(f)$.
\end{compactenum}

\begin{dfn} \label{D:3}
Let $D$ be an effective divisor on $\Gamma$ and $S$ a subset of
$\text{supp}D\setminus v_0$ such that all the boundary points of
$\mathcal{U}_{S,v_0}^c$ are saturated with respect to $D$. Let $\Omega$ be a fixed vertex set of $\Gamma$. We call
the following parameterizing process
$\Delta_{D,S,v_0}:[0,1]\rightarrow \text{Div}_+\Gamma$ the
$v_0$-\emph{move} of $D$ with respect to $S$ and $\Omega$:
\begin{compactenum}[(i)]
\item $\Delta_{D,S,v_0}^{(0)}=D$.
\item Let $J$ be the number of connected components of
$\mathcal{U}_{S,v_0}^c$, and denote these components by $X_1$
through $X_J$.

For $j=1,2,\cdots,J$ and $t\in(0,1]$, let\\
$d_j^{(t)}=t\cdot \textrm{dist} \big(X_j,
\mathcal{U}_{S,v_0} \bigcap (\Omega \bigcup v_0) \big)$,\\
$P_j^{(t)}=\{p \in \mathcal{U}_{S,v_0}\mid
\textrm{dist}(X_j,p)=d_j^{(t)}\}$, \\
$Q_j^{(t)}=\{q \in \mathcal{U}_{S,v_0}\mid \textrm{dist}(X_j,q)\leqslant
d_j^{(t)} \}$, and\\
$f_j^{(t)}$ a basic $v_0$-extremal function such that\\
$X_{\text{max}}(f_j^{(t)})=X_j$, and $\partial
X_{\text{min}}(f_j^{(t)})=P_j^{(t)}$.
\item $\Delta_{D,S,v_0}^{(t)}=D+\sum_{j=1}^J(f_j^{(t)})$, for $t\in(0,1]$.
\end{compactenum}
\end{dfn}

\begin{ex}
\begin{figure}[h]
\centering
\includegraphics[width=1\textwidth]{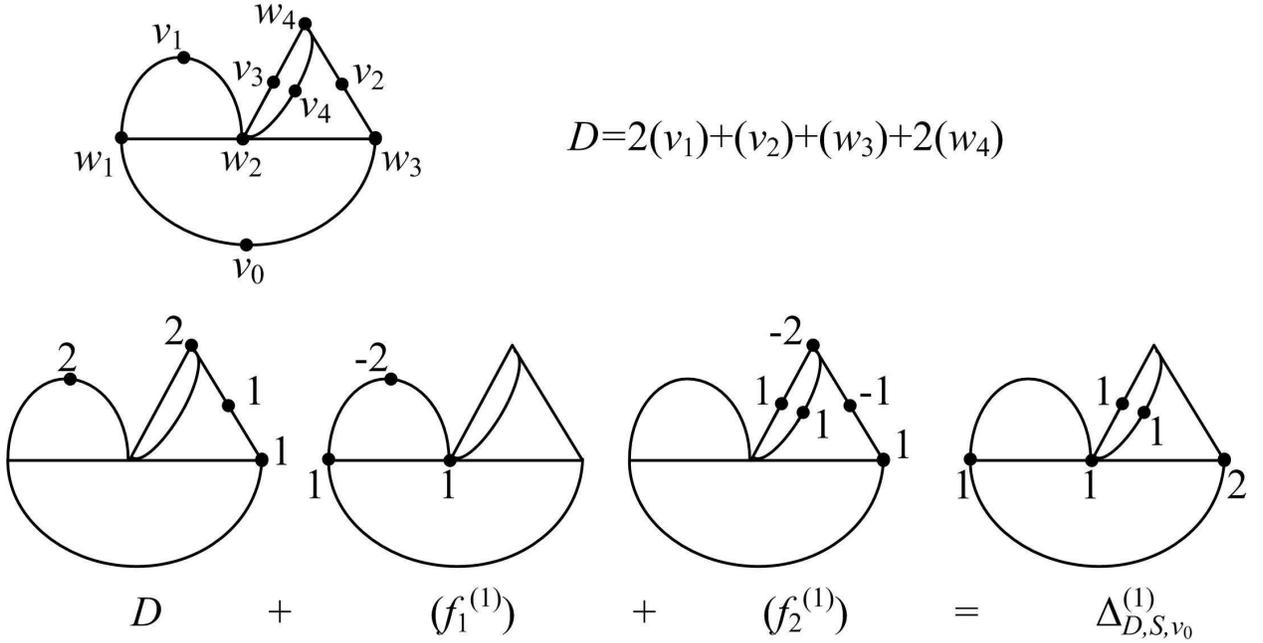}
\caption{A $v_0$-move of $D$.}\label{F:move}
\end{figure}
Let $\Gamma$ be the same metric graph as in Example~\ref{E:Dhar} and
$D=D_2$, as shown in Figure~\ref{F:move}. In particular, we assign
length $1$ to all edges and let $v_i$ be the middle point of the
corresponding edge for $i=0,1,2,3,4$. We know from Example~\ref{E:Dhar} that the output $S$ of Dhar's algorithm for $D$ and
$v_0$ is $\{v_1,v_2,w_4\}$. Let us consider a $v_0$-move
$\Delta_{D,S,v_0}$. Note that $\mathcal{U}_{S,v_0}^c$ has two
connected components, $v_1$ and $[v_2,w_4]$, which we denote by $X_1$ and
$X_2$ respectively. We observe that $d_1^{(t)}=d_2^{(t)}=0.5t$ for
$t\in(0,1]$. And at the end of the move ($t=1$), we get
$P_1^{(1)}=\{w_1,w_2\}$, $Q_1^{(1)}=[w_1,v_1,w_2]\setminus v_1$,
$P_2^{(1)}=\{v_3,v_4,w_3\}$, and
$Q_2^{(1)}=(w_4,v_3]\bigcup(w_4,v_4]\bigcup(v_2,w_3]$. In addition,
$(f_1^{(1)})=(w_1)+(w_2)-2(v_1)$ and
$(f_2^{(1)})=(v_3)+(v_4)+(w_3)-(v_2)-2(w_4)$. Then we get
$\Delta_{D,S,v_0}^{(1)}=D+(f_1^{(1)})+(f_2^{(1)})=(v_3)+(v_4)+(w_1)+(w_2)+2(w_3)$.
\end{ex}

\begin{lem} \label{L:1:3}
Let $D$ be an effective divisor which is zero at $v_0$ and
$\Delta_{D,S,v_0}$ a move of $D$. Denote
$\emph{supp}\Delta_{D,S,v_0}^{(t)}$ by $O^{(t)}$ for $t\in[0,1]$.
Then $\mathcal{U}_{O^{(t)},v_0}$ is non-expanding with respect to
$t$. Moreover, $\mathcal{U}_{O^{(t)},v_0}$ evolves continuously
unless possibly undergoing an abrupt shrink at $t=1$.
\end{lem}

\begin{proof}
Let $Q_j^{(t)}$ be as defined in Definition~\ref{D:3} for
$t\in(0,1]$. Let $Q^{(0)}=\partial\mathcal{U}_{S,v_0}$ and
\begin{equation*}
Q^{(t)}=\bigcup_{j=1}^J Q_j^{(t)}\text{, for }t\in(0,1].
\end{equation*}
Clearly, $Q^{(t)}$ continuously expands with respect to $t$. For
$t\in[0,1)$, we have
\begin{equation*}
\mathcal{U}_{O^{(t)},v_0}=\mathcal{U}_{O^{(0)},v_0} \setminus
Q^{(t)},
\end{equation*}
which means $\mathcal{U}_{O^{(t)},v_0}$ is non-expanding as $t$
increases and its evolution is continuous. The case $t=1$ is somehow
special, since the continuous expansion of $Q(t)$ might result in a hit
at certain vertices or $v_0$. But we still have
\begin{equation*}
\mathcal{U}_{O^{(1)},v_0}\subseteq\mathcal{U}_{O^{(0)},v_0}
\setminus Q^{(1)}.
\end{equation*}
This means that an abrupt shrink of $\mathcal{U}_{O^{(t)},v_0}$
might happen at $t=1$.
\end{proof}

Based on making $v_0$-moves iteratively, we propose the following
algorithm to derive the $v_0$-reduced divisor linearly equivalent to
an effective divisor $D$.

\begin{algm} \label{A:main}
\textbf{Input:} An effective divisor $D\in \text{Div}_+\Gamma$, and a point $v_0\in\Gamma$.\\
\textbf{Output:} The unique $v_0$-reduced divisor $D_{v_0}$ linearly equivalent to $D$.\\
Initially, set $D^{(0)}=D$, and $i=0$.
\begin{compactenum}[(1)]
\item Run Dhar's algorithm for $D^{(i)}$ and $v_0$ with the output denoted by $S^{(i)}$. If $S^{(i)}=\emptyset$, then set $D_{v_0}=D^{(i)}$ and stop the
procedure. In addition, we say that the procedure terminates at $i$.
And for convenience, we set $D^{(t)}=D^{(i)}$ for all real numbers
$t>i$. Otherwise, go to step (2).
\item Make the $v_0$-move $\Delta_{D^{(i)},S,v_0}$ of $D^{(i)}$ with respect to $S^{(i)}$. Let
$D^{(i+t)}=\Delta_{D^{(i)},S,v_0}^{(t)}$ for $t\in(0,1]$. Set
$i\leftarrow i+1$, and go to step (1).
\end{compactenum}
\end{algm}

If the procedure in Algorithm~\ref{A:main} terminates at $I$, then
by Lemma~\ref{L:1:2}, $D_{v_0}$ is $v_0$-reduced as desired, and the
evolution of $D$ into $D_{v_0}$ is parameterized by $D^{(t)}$,
$t\in[0,I]$. The main goal of this section is to prove such a
procedure always terminates (Theorem~\ref{T:tmt}), which means that we will always get to
a reduced divisor using finitely many moves.

\begin{lem} \label{L:1:4}
We have the following properties of the parameterizing procedure in
Algorithm~\ref{A:main}:
\begin{compactenum}[(i)]
\item $D^{(t)}(v_0)$ is integer-valued, bounded, and
non-decreasing with respect to $t$, and it can jump only when $t$ is
an integer. In addition, there exists an integer $I_1$ such that
$D^{(t)}(v_0)=D^{(I_1)}(v_0)$ for all $t\geqslant I_1$.

\item For a non-negative integer $i_0$, let $d=D^{(i_0)}(v_0)$ and $D_0^{(t)}=D^{(t)}-d\cdot(v_0)$. Then
for all real numbers $t\geqslant i_0$,
$\mathcal{U}_{\emph{supp}D_0^{(t)},v_0}$ is non-expanding with
respect to $t$. In particular,
$\mathcal{U}_{\emph{supp}D_0^{(t)},v_0}$ evolves continuously unless
possibly undergoing an abrupt shrink when $t$ is an integer.

\item Denote $\mathcal{U}_{\emph{supp}D^{(t)}\setminus
v_0,v_0}$ by $U(t)$. For $t\geqslant I_1$, let
$K^{(t)}=\#\{\Omega\bigcap U(t)\}$, which counts the number of
vertices in $U(t)$ after $D^{(t)}(v_0)$ reaches its maximum. Then
$K^{(t)}$ is integer-valued, bounded, and non-increasing with
respect to $t$, and it can jump only when $t$ is an integer.
Furthermore, there exists an integer $I_2 \geqslant I_1$ such that
$K^{(t)}=K^{(I_2)}$ for all $t\geqslant I_2$.

\end{compactenum}
\end{lem}

\begin{proof}
Clearly $D^{(t)}(v_0)$ is integer-valued. Note that $v_0 \notin
S^{(i)}$ for any $i$, which implies that $D^{(t)}(v_0)$ is
non-decreasing and can only change its value when $t$ is an
integer. Moreover, $D^{(t)}(v_0)$ is bounded from below by $D(v_0)$
and from above by $\deg{(D)}$, which guarantees the existence of the
finite integer $I_1$. Thus Property \mbox{(i)} holds.

$D_0^{(i_0)}$ has value $0$ at $v_0$. Thus by Lemma~\ref{L:1:3}, for
$t\geqslant i_0$, $\mathcal{U}_{\text{supp}D_0^{(t)},v_0}$ is
non-expanding, and evolves continuously unless possibly undergoing
an abrupt shrink when $t$ is an integer. In particular, whenever
$v_0$ is hit by a move, $\mathcal{U}_{\text{supp}D_0^{(t)},v_0}$
will always be empty afterwards. And Property \mbox{(ii)} is proved.

After $D^{(t)}(v_0)$ reaches its maximum at $t=I_1$, $v_0$ will
never be hit anymore. The above argument implies that for
$t\geqslant I_1$, $U(t)$ is non-expanding, and continuously evolves
unless possibly undergoing an abrupt shrink when $t$ is an integer.
It follows immediately that $K^{(t)}$ is integer-valued, and
non-increasing with respect to $t$, while it only possibly changes
when $t$ is an integer. Clearly $K^{(t)}$ is lower-bounded by 0,
which also implies the existence of $I_2$ and finishes the proof of
Property \mbox{(iii)}.
\end{proof}

\begin{thm} \label{T:tmt}
The procedure in Algorithm~\ref{A:main} always terminates.
\end{thm}
\begin{proof}
We proceed by induction on $\deg{(D)}$. Clearly Theorem~\ref{T:tmt}
holds when $\deg{D}=0$ since this implies that $D=0$. Now suppose
$\deg{(D)}>0$.

By Lemma~\ref{L:1:4}\mbox{(i)}, if $D^{(I_1)}(v_0)>0$, then $D^{(t)}(v_0)>0$ for all $t\geqslant 0$ and the result follows by induction (applied to$D^{(I_1)}-(v_0)$). Now we assume $D^{(I_1)}(v_0)=0$. By Lemma~\ref{L:1:4}\mbox{(iii)}, there exists an integer $I_2$, such that
$K^{(t)}=K^{(I_2)}$ for all $t\geqslant I_2$. We let $t\geqslant
I_2$ in the remaining parts of the proof. Note that $U(t)$ might
keep shrinking. However, such a shrink can never hit a vertex
anymore, which also means that $U(t)$ evolves continuously for
$t\geqslant I_2$. Let $X$ be a connected component of $U(I_2)^c$.
Let $U_0$ be a subset of $U(I_2)$ derived by removing the interior
of all the segments with one end open and the other end a vertex or
$v_0$. Clearly $U_0$ is closed
and connected. And $U(I_2)\setminus U_0$ is a union of some disjoint
open segments. Denote by $\mathcal{E}_X$ the set of these segments.
For $e\in\mathcal{E}_X$, we use $w_e$ to denote the end of $e$ on
$X$. We say $e\in\mathcal{E}_X$ is \emph{obstructed} at $t$ if
$\text{supp}D^{(t)}\bigcap e\neq\emptyset$ or $w_e$ is saturated
with respect to $D^{(t)}$ and $X$. Note that if an edge is
obstructed at $t$, then it is obstructed at all $t'\geqslant t$.

We claim that there exists $e\in\mathcal{E}_X$ that never becomes
obstructed. Otherwise, there exists an integer $I_3$ such that for
$t\geqslant I_3$, the component of $U(t)^c$ corresponding to $X$ has
all its boundary points saturated. Then one additional move from
Algorithm~\ref{A:main} will result in a hit at a vertex, which
contradicts the minimality of $K^{(I_2)}$. So let $e$ be
an element of $\mathcal{E}_X$ that never becomes obstructed. Then
$w_e$ does not belong to any output $S^{(i)}$ of Dhar's algorithm
for $D^{(i)}$ when $i\geqslant I_2$. So Algorithm~\ref{A:main} for
$D^{(I_2)}$ terminates if and only if the algorithm for $D^{(I_2)}-(w_e)$
terminates, and
 the induction applies.
\end{proof}

\begin{rmk}
What should $X$ look like in the above proof? Since $X$ must contain
non-saturated boundary points with respect to $D^{(I_2)}$, there are
only two possibilities. $X$ can be a single non-vertex point with
$D^{(I_2)}(X)=1$, or else $X^{(I_2)}$ must contain a
vertex on its boundary.
\end{rmk}

\begin{rmk}
We know from Riemann-Roch theorem that the rank of the divisor $n\cdot(v_0)$ as a function of $n$ can be arbitrarily large. Hence given a divisor $D$ (not necessarily effective) on $\Gamma$, there always exists a divisor $D'$ which is non-negative on $\Gamma\setminus v_0$ and linearly equivalent to $D$. In particular, \cite{HKN07} presents an algorithm to construct such a divisor $D'$ as the first step in the proof of the existence part of Theorem~\ref{T:unique} (Theorem~10 in \cite{HKN07}).
By running Algorithm~\ref{A:main} for $D'-D'(v_0)\cdot(v_0)$ and $v_0$, we can always
obtain a $v_0$-reduced divisor $D''$ linearly equivalent to
$D-D'(v_0)\cdot(v_0)$. Then $D''+D'(v_0)\cdot(v_0)$ is a $v_0$-reduced divisor linearly equivalent to $D$. This
provides an alternative proof of the existence part of Theorem~\ref{T:unique}.
\end{rmk}

\begin{cor} \label{C:1:1}
Let $D$ be a divisor on $\Gamma$ and $|D|$ the linear system
associated to $D$. For $v_0\in\Gamma$, let $D_{v_0}$ be the unique
$v_0$-reduced divisor $D_{v_0}$ in $|D|$.
\begin{compactenum}[(i)]
\item If $v_0\in\emph{supp}|D|$, then $D_{v_0}(v_0)>0$.
\item If $|D|\neq\emptyset$ and $v_0\notin\emph{supp}|D|$, then $\mathcal{U}_{\emph{supp}D_{v_0},v_0}$ is nonempty and for
all $v\in\mathcal{U}_{\emph{supp}D_{v_0},v_0}$, we have
$v\notin\emph{supp}|D|$ and $D_{v_0}$ is also $v$-reduced.
\end{compactenum}
\end{cor}

\begin{proof}
If $v_0\in\text{supp}|D|$, let $D'$ be an effective divisor such
that $D'\in|D|$ and $D'(v_0)>0$. Applying Algorithm~\ref{A:main} for
$D'$ and $v_0$, we can derive $D_{v_0}$. Note that
$D_{v_0}(v_0)\geqslant D'(v_0)$. Thus $D_{v_0}(v_0)>0$.

If $|D|\neq\emptyset$ and $v_0\notin\text{supp}|D|$, then
$D_{v_0}(v_0)=0$, which means
$\mathcal{U}_{\text{supp}D_{v_0},v_0}$ is nonempty. For all
$v\in\mathcal{U}_{\text{supp}D_{v_0},v_0}$, clearly $D_{v_0}(v)=0$,
and using Dhar's algorithm, it is easy to see that $D_{v_0}$ is
also $v$-reduced . Moreover, we have $v\notin\text{supp}|D|$ by
\mbox{(i)}.
\end{proof}

\begin{rmk}
In the sense of Corollary~\ref{C:1:1}\mbox{(ii)}, if $X$ is a subset
of $\mathcal{U}_{\text{supp}D_{v_0},v_0}$, then we may also say
$D_{v_0}$ is $X$-reduced.
\end{rmk}

\section{Rank-determining sets} \label{S:2}

We say a subset $\Gamma'$ of a metric graph $\Gamma$ is a
\emph{subgraph} of $\Gamma$ if $\Gamma'$ is connected and closed.
Let $\Omega$ be a vertex set of $\Gamma$. Then
$(\Omega\bigcap\Gamma')\bigcup\partial\Gamma'$ (considered in
$\Gamma$) is automatically a vertex set of $\Gamma'$, which we call
the vertex set of $\Gamma'$ induced by $\Gamma$. A \emph{tree} on
$\Gamma$ is a subgraph of $\Gamma$ with genus $0$, and a
\emph{spanning tree} of $\Gamma$ is a tree on $\Gamma$ that is
minimal among those which contain all vertices of $\Gamma$. We call a point $v$
a \emph{cut point} in a metric graph if $\Gamma\setminus v$ is
disconnected.

\subsection{$A$ is a rank-determining set if and only if $\mathcal{L}(A)=\Gamma$}
For a nonempty subset $A$ of $\Gamma$, we use $\mathcal{L}(A)$ to
denote a subset of $\Gamma$ such that $v\in\mathcal{L}(A)$ if and only if
$A\subseteq \text{supp}|D|$ implies $v\in\text{supp}|D|$. For
simplicity of notation, we denote $\mathcal{L}(\bigcup_{i=1}^nA_i)$ by
writing $\mathcal{L}(A_1,A_2,\cdots,A_n)$. Note that we can always
find a linear system whose support contains $A$ (for example, the
support of the linear system associated to $\sum_{v\in\Omega}(v)$ is
the whole graph $\Gamma$). Therefore
\begin{equation*}
\mathcal{L}(A)=\bigcap_{\text{supp}|D|\supseteq A}\text{supp}|D|.
\end{equation*}
Obviously, $A\subseteq\mathcal{L}(A)$, and if $A'$ is a subset of
$\mathcal{L}(A)$, then $\mathcal{L}(A,A')=\mathcal{L}(A)$. In case
we want to emphasize that $A$ and all the linear systems are defined on
$\Gamma$, we may write $\mathcal{L}_\Gamma(A)$ in stead of
$\mathcal{L}(A)$.

\begin{prop} \label{P:2:1}
Let $A$ be a nonempty subset of $\Gamma$. The following are
equivalent.
\begin{compactenum}[(i)]
\item $\mathcal {L}(A)=\Gamma$.
\item If $r_A(D)\geqslant 1$, then $r(D)\geqslant 1$.
\item $A$ is a
rank-determining set of $\Gamma$.
\end{compactenum}
\end{prop}

\begin{proof}
\mbox{(i)}$\Leftrightarrow$\mbox{(ii)}. $\mathcal {L}(A)=\Gamma$,
if and only if $A\subseteq\text{supp}|D|$ implies $\text{supp}|D|=\Gamma$, if and only if
$|D-E'_1|\neq\emptyset$ for all $E'_1\in \text{Div}_+^1A$, implies
$|D-E_1|\neq\emptyset$ for all $E_1\in \text{Div}_+^1\Gamma$, if and only if
$r_A(D)\geqslant 1$ implies $r(D)\geqslant 1$.

\mbox{(iii)}$\Rightarrow$\mbox{(ii)}. This follows directly from the definition
of rank-determining sets.

\mbox{(ii)}$\Rightarrow$\mbox{(iii)}. If $|D|=\emptyset$, then
$r_A(D)=r(D)=-1$. We will only consider the case $|D|\neq\emptyset$
in the following. Since $A$ is a subset of $\Gamma$, it is easy to
see that $r_A(D)\geqslant r(D)$ by definition. Therefore, to prove
$A$ is a rank-determining set, it suffices to show that
$r_A(D)\geqslant s$ implies $r(D)\geqslant s$ for each integer
$s\geqslant 0$. The case $s=0$ is trivial, since
$\text{Div}_+^0A=\text{Div}_+^0\Gamma={0}$. And the case $s=1$ is
stated in \mbox{(ii)}.

We claim that $r_A(D-E_k)\geqslant s-k$, $\forall
E_k\in\text{Div}_+^k\Gamma$, implies $r_A(D-E_{k+1})\geqslant
s-k-1$, $\forall E_{k+1}\in\text{Div}_+^{k+1}\Gamma$, for
$s\geqslant 0$ and $k=0,1,\cdots,s-1$. Equivalently, it says
$|D-E_k-E'_{s-k}|\neq\emptyset$, $\forall
E_k\in\text{Div}_+^k\Gamma$, $\forall
E'_{s-k}\in\text{Div}_+^{s-k}A$, implies
$|D-E_{k+1}-E'_{s-k-1}|\neq\emptyset$, $\forall
E_{k+1}\in\text{Div}_+^{k+1}\Gamma$, $\forall
E'_{s-k-1}\text{Div}_+^{s-k-1}A$, for $s\geqslant 0$ and
$k=0,\cdots,s-1$. This can be proved by the following deduction:

\begin{equation*}
r_A(D-E_k)\geqslant s-k, \quad\forall E_k\in\text{Div}_+^k\Gamma
\end{equation*}
\begin{equation*}
\Longleftrightarrow
\end{equation*}
\begin{equation*}
|D-E_k-E'_{s-k}|\neq\emptyset, \quad\forall
E_k\in\text{Div}_+^k\Gamma,
\quad\forall E'_{s-k}\in\text{Div}_+^{s-k}A
\end{equation*}
\begin{equation*}
\Longleftrightarrow
\end{equation*}
\begin{equation*}
|(D-E_k-E'_{s-k-1})-E'_1|\neq\emptyset,\quad
\quad\forall E_k\in\text{Div}_+^k\Gamma, \quad\forall
E'_{s-k-1}\in\text{Div}_+^{s-k-1}A,\quad\forall E'_1\in
\text{Div}_+^1A
\end{equation*}
\begin{equation*}
\text{(By \mbox{(ii)})} \Longrightarrow
\end{equation*}
\begin{equation*}
|(D-E_k-E'_{s-k-1})-E_1|\neq\emptyset,\quad
\quad\forall E_k\in\text{Div}_+^k\Gamma,\quad\forall
E'_{s-k-1}\in\text{Div}_+^{s-k-1}A,\quad\forall E_1\in
\text{Div}_+^1\Gamma
\end{equation*}
\begin{equation*}
\Longleftrightarrow
\end{equation*}
\begin{equation*}
|D-E_{k+1}-E'_{s-k-1}|\neq\emptyset, \quad\forall
E_{k+1}\in\text{Div}_+^{k+1}\Gamma,\quad\forall
E'_{s-k-1}\in\text{Div}_+^{s-k-1}A
\end{equation*}
\begin{equation*}
\Longleftrightarrow
\end{equation*}
\begin{equation*}
r_A(D-E_{k+1})\geqslant s-k-1, \quad\forall
E_{k+1}\in\text{Div}_+^{k+1}\Gamma.
\end{equation*}

Therefore, by applying the above deduction for $k$ going from
$0$ through $s-1$, we have:
\begin{align*}
r_A(D)\geqslant s \qquad \Longrightarrow\\
r_A(D-E_1)\geqslant s-1, \quad\forall
E_1\in\text{Div}_+^1\Gamma\qquad \Longrightarrow\\
\cdots\qquad \Longrightarrow\\
r_A(D-E_{s-1})\geqslant 1, \quad\forall
E_{s-1}\in\text{Div}_+^{s-1}\Gamma\qquad \Longrightarrow\\
r_A(D-E_s)\geqslant 0, \quad\forall
E_s\in\text{Div}_+^s\Gamma\qquad \Longleftrightarrow\\
r(D)\geqslant s. \qquad\qquad
\end{align*}

Thus \mbox{(ii)} is sufficient to make $A$ a rank-determining set of
$\Gamma$.
\end{proof}

\subsection{Special open sets and a criterion for $\mathcal{L}(A)$}

\begin{dfn}
A connected open subset $U$ of $\Gamma$ is called a \emph{special
open set} on $\Gamma$ if either $U=\emptyset$ or $\Gamma$, or every
connected component $X$ of $U^c$ contains a boundary point $v$ such
that $\text{outdeg}_X(v)\geqslant 2$. In particular, we say $\Gamma$
is \emph{trivial} if $U=\emptyset$ or $\Gamma$.  And we use
$\mathcal{S}_\Gamma$ to denote the set of all special open sets on
$\Gamma$.
\end{dfn}

Lemma~\ref{L:2:1} through \ref{L:2:5} present some simple properties
of special open sets.

\begin{lem} \label{L:2:1}
Let $U$ be a connected open set on $\Gamma$, and
$D=\sum_{v\in\partial U}(v)$. Then $U$ is a special open set if and only if $D$
is $U$-reduced.
\end{lem}
\begin{proof}
We just need to consider $U$ nontrivial. And it follows directly by
running Dhar's algorithm for $D$ and any point $v\in U$.
\end{proof}

\begin{lem} \label{L:2:2}
For $v_0\in\Gamma$, if $D$ is a $v_0$-reduced divisor, then
$\mathcal{U}_{\emph{supp}D\setminus v_0,v_0}$ is a special open set.
\end{lem}
\begin{proof}
Let $D'=\sum_{v\in\text{supp}D\setminus v_0}(v)$. Since $D$ is a
$v_0$-reduced divisor, $D'$ must also be $v_0$-reduced. Thus
$\mathcal{U}_{\text{supp}D\setminus v_0,v_0}$ is a special open set
by Lemma~\ref{L:2:1}.
\end{proof}

\begin{lem} \label{L:2:3}
Let $\Gamma$ be a metric graph of genus $g$. If $U$ is a nontrivial special open set on $\Gamma$, then
$\overline{U}$ has genus at least $1$. In addition, there exist at most $g$ disjoint nonempty special open sets on $\Gamma$.
\end{lem}
\begin{proof}
If $\overline{U}$ is a tree, then for every $v\in\partial U$,
$\text{outdeg}_{U^c}(v)=1$, which contradicts the definition of
special open sets. And it follows immediately that $\Gamma$ can sustain at most $g$ disjoint nonempty special open set.
\end{proof}

\begin{lem} \label{L:2:4}
Let $X$ be a nonempty connected subset of $\Gamma$, and $|D|$ a
linear system such that $\emph{supp}|D|\bigcap X=\emptyset$. Then
there exists a special open set $U$ such that $X\subseteq
U\subseteq(\emph{supp}|D|)^c$.
\end{lem}
\begin{proof}
Let $v\in X$ and $D'$ be the $v$-reduced divisor in $|D|$. Then by
Corollary~\ref{C:1:1} and Lemma~\ref{L:2:2},
$\mathcal{U}_{\text{supp}_{D'},v}$ is a special open set with the desired properties.
\end{proof}

\begin{lem} \label{L:2:5}
Let $D$ be a divisor on $\Gamma$ and $|D|$ the corresponding
linear system. Then $(\emph{supp}|D|)^c$ is a disjoint union of finitely many nonempty special open sets.
\end{lem}
\begin{proof}
Let $v_1$ and $v_2$ be two points in $(\text{supp}|D|)^c$. Let $D_1$ and $D_2$ be elements of $|D|$ that are $v_1$-reduced and $v_2$-reduced, respectively. Let $U_1=\mathcal{U}_{\text{supp}D_1,v_1}$ and $U_2=\mathcal{U}_{\text{supp}D_2,v_2}$.
Then by Lemma~\ref{L:2:2}, $U_1$ and $U_2$ are special open sets. In addition, we have either $U_1=U_2$ or $U_1\bigcap U_2=\emptyset$ by Corollary~\ref{C:1:1}. Thus $(\text{supp}|D|)^c$ must be a disjoint union of nonempty special open sets. And we know from Lemma~ \ref{L:2:3} that there are only finitely many of them.
\end{proof}

Based on the notion of special open sets, we formulate a sufficient
condition for $v$ to belong to $\mathcal{L}(A)$, as stated in the following
theorem. (We will show in Theorem~\ref{T:criterion} that it
is also a necessary condition.)
\begin{thm} \label{T:precriterion}
Let $v\in \Gamma$ and let $A$ be a nonempty subset of $\Gamma$. Then
$v\in\mathcal{L}(A)$ if for all special open sets $U$ containing
$v$, we have $A\bigcap U\neq \emptyset$. Moreover,
\begin{equation*}
\mathcal{L}(A)\supseteq \bigcap_{U\in\mathcal{S}_\Gamma,U\bigcap
A=\emptyset} U^c.
\end{equation*}
In addition, $A$ is a rank-determining set if all nonempty special
open sets intersect $A$.
\end{thm}
\begin{proof}
Suppose $|D|$ is a linear system such that
$A\subseteq\text{supp}|D|$. Then by Lemma~\ref{L:2:4}, for every
$v\notin\text{supp}|D|$, there exists a neighborhood $U$ of $v$ which is a special open set disjoint from $\text{supp}|D|$. Thus if all
special open sets containing $v$ intersect $A$, then
$A\subseteq\text{supp}|D|$ implies $v\in\text{supp}|D|$, which means
$v\in\mathcal{L}(A)$. It follows immediately that
\begin{equation*}
\mathcal{L}(A)\supseteq \bigcap_{U\in\mathcal{S}_\Gamma,U\bigcap
A=\emptyset} U^c.
\end{equation*}

If all nonempty special open sets intersect $A$, then
$\mathcal{L}(A)=\Gamma$. Thus $A$ is a rank-determining set by
Proposition~\ref{P:2:1}.
\end{proof}

\begin{prop} \label{P:2:2}
Let $U$ be a nonempty connected open proper subset of $\Gamma$ such
that $\overline{U}$ is a tree. Then $\overline{U}\subseteq
\mathcal{L}(\partial U)$.
\end{prop}
\begin{proof}
$\partial U$ is nonempty since $U$ is a proper subset of $\Gamma$.
Then by Lemma~\ref{L:2:3}, for every $v\in U$, if $U'$ is a critical
open set containing $v$, then $\overline{U'}$ has genus at least $1$
unless possibly $U'$ is the whole graph. Thus $U'$ must intersect
$\partial U$, since any connected closed subset of $\overline{U}$
has genus 0. Therefore we have $v\in\mathcal{L}(\partial U)$ by
Theorem~\ref{T:precriterion}.
\end{proof}

\begin{ex}
(a) By Proposition~\ref{P:2:2}, we immediately have
$[w_i,w_j]\subseteq\mathcal{L}(w_i,w_j)$ for two adjacent vertices
$w_i$ and $w_j$ (note that it doesn't matter whether there are
multiple edges between $w_i$ and $w_j$). Thus
$\mathcal{L}(\Omega)=\Gamma$, which implies $\Omega$ is a
rank-determining set of $\Gamma$, as claimed in \textbf{Theorem~\ref{T:vertex}}.

(b) Let $A$ be a finite set formed by choosing one internal point from each
edge. Then it is also easy to show that $A$ is a rank-determining
set using Proposition~\ref{P:2:2}.
\end{ex}

\begin{prop} \label{P:2:3}
Let $U$ be a nonempty connected open proper subset of a metric graph
$\Gamma$ such that $\overline{U}$ has genus $g'$. Let $T$ be a
spanning tree of $\overline{U}$. Then $U\setminus T$ is a disjoint
union of $g'$ open segments. Choosing one point from each of these
segments, we get a finite set $B$ of cardinality $g'$. Then
$\overline{U}\subseteq \mathcal{L}(\partial U, B)$
\end{prop}
\begin{proof}
If $g'=0$, then $\overline{U}\subseteq \mathcal{L}(\partial U)$ by
Proposition~\ref{P:2:2}. Now we suppose $g'\geqslant 1$. Consider a
point $v\in U$. If $v\notin \mathcal{L}(\partial U)$, then there
exists a special open set $U'$ such that $v\in U'$ and $U'\subseteq
U$ by Theorem~\ref{T:precriterion}. We claim that $U'\bigcap B\neq
\emptyset$, which implies $v\in\mathcal{L}(\partial U, B)$.

Denote the $g'$ open segments of $U\setminus T$ by
$e_1,e_2,\cdots,e_{g'}$. If $U'\bigcap T$ is not connected, then
there must exist some $e_i\subseteq U'\setminus T$ to make $U'$
connected. Thus $U'\bigcap B\neq \emptyset$. Now suppose $U'\bigcap
T$ is connected. By definition of special open sets, every connected
component of $(U')^c$ contains a boundary point with out-degree at
least $2$, which means that there exists some $e_i\subseteq U'\setminus
T$ having one end in $\partial U'$ and the other in $U'\bigcap
T$. Thus we also have $U'\bigcap B\neq \emptyset$.
\end{proof}

\begin{rmk}
\textbf{Theorem~\ref{T:g+1}} can be deduced from Proposition~\ref{P:2:3} by the following argument. Let $\Gamma$ be a metric
graph of genus $g$ and $T$ a spanning tree of $\Gamma$. Choose an
arbitrary point $v_0$ from $T$. Then $\Gamma\setminus T$ is a
disjoint union of $g$ open segments $e_1,e_2,\cdots,e_{g}$. Choose
arbitrarily a point $v_i$ from $e_i$ for $i=1,2,\cdots,g$. Let
$A=\{v_0,v_1,\cdots,v_g\}$. If $v_0$ is not a cut point, then we can
directly apply Proposition~\ref{P:2:3} to $\Gamma\setminus v_0$ and
conclude that $\mathcal{L}(A)=\Gamma$. Otherwise, applying
Proposition~\ref{P:2:3} to each connected component $X$ of
$\Gamma\setminus v_0$ (note that the induced spanning tree of
$\overline{X}$ is $T\bigcap\overline{X}$), we also get
$\mathcal{L}(A)=\Gamma$. Therefore $A$ is a rank-determining set of
cardinality $g+1$ as desired.
\end{rmk}

\begin{rmk} \label{R:ac}
We sketch Varley's proof of \textbf{Theorem~\ref{T:ac}} here. Consider a
nonsingular projective algebraic curve $C$. First note that the rank $r(D)$ of
a divisor $D$ on $C$ has the same value as $\text{dim}L(D)-1$.
Recall that we say a point $p\in C$ is a \emph{base point} of a
linear system $|D|$ if $p$ belongs to the support of every element
of $|D|$, i.e., $p\in \text{BL}(|D|)$ where
$\text{BL}(|D|)=\bigcap_{D'\in |D|}\text{supp}D'$ which is called
the \emph{base locus} of $|D|$. Varley's argument uses the fact that
a point $p\in C$ is a base point of $|D|$ if and only if $r(D-(p))=r(D)$. (Note
that this is not true for metric graphs.) Take any set $S$ of $g +
1$ distinct points on $C$. To prove that $S$ is a rank-determining set,
it suffices to show that for a divisor $D$ on $C$, if $r(D)
\geqslant 0$, then there exists a point $p$ in $S$ such that $r(D-(p))
= r(D) - 1$. Let $B=\sum_{q\in\text{BL}(|D|)}(q)$ which is the full
base locus divisor of $|D|$. Note that $|B|=\{B\}$ since $B$ cannot ``move''. If $\text{deg}(B)\leqslant g$,
then there is a point $p$ of $S$ not contained in $\text{BL}(|D|)$,
which means $r(D-(p)) = r(D) - 1$.  If $\text{deg}(B)\geqslant g+1$,
then $r(B)\geqslant 1$ (by Riemann-Roch) which is impossible.  The
desired result follows by induction.
\end{rmk}

\begin{ex} \label{E:K4}
\begin{figure}[h]
\centering
\includegraphics[width=0.4\textwidth]{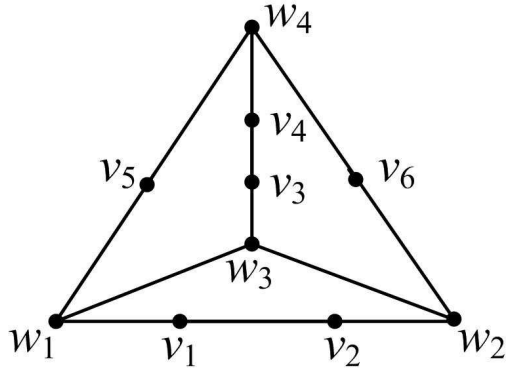}
\caption{A metric graph corresponding to $K_4$.}\label{F:K4}
\end{figure}
Let $\Gamma$ be a metric graph corresponding to $K_4$ with a vertex
set $\Omega$ being $\{w_1,w_2,w_3,w_4\}$ as shown in Figure~\ref{F:K4}. Let $v_1,v_2,\cdots,v_6$ be some internal points.
Clearly $\Omega$ itself is a rank-determining set by Theorem~\ref{T:vertex}. But a proper subset of $\Omega$ can also be a
rank-determining set. Note that
$[w_1,w_3]\bigcup[w_2,w_3]\bigcup[w_4,w_3]$ is a spanning tree of
$\Gamma$, which implies $w_3\in\mathcal{L}(w_1,w_2,w_4)$ by
Proposition~\ref{P:2:2}. Thus $\{w_1,w_2,w_4\}$ is a
rank-determining set as desired. It is also easy to see that
$\{w_3,v_1,v_5,v_6\}$ and $\{v_1,v_3,v_5,v_6\}$ are rank-determining
sets by Proposition~\ref{P:2:3}. We recommend the reader to use Theorem~\ref{T:precriterion} to verify
that $\{v_1,v_2,v_3,v_4\}$ is another rank-determining set, which is
not obvious at first sight.
\end{ex}

\begin{prop} \label{P:2:4}
Let $U$ be a special open set on $\Gamma$. Then there exists a
divisor $D$ such that $\emph{supp}|D|=U^c$.
\end{prop}
\begin{proof}
We only need to consider $U$ nontrivial. Assume $(\partial U)^c$ has
$n$ connected components $X_1,X_2,\cdots,X_n$ other than $U$. Let
$T_i$ be a spanning tree of $\overline{X}_i$, $i=1,2,\cdots,n$. Then
$X_i\setminus T_i$ is a disjoint union of $g_i$ open segments.
Choosing one point from each of these segments, we get a finite set
$B_i$ of cardinality $g_i$. Let $B=\bigcup_{i=1}^nB_i$ and
$D=\sum_{v\in\partial U}(v)+\sum_{v\in B}(v)$. Then by Proposition~\ref{P:2:3}, we have $U^c=\bigcup_{i=1}^n\overline{X}_i\subseteq
\mathcal{L}(\partial U, B)\subseteq\text{supp}|D|$. Therefore, to
prove $\text{supp}|D|=U^c$, it suffices to show that $D$ is
$U$-reduced.

Let $D'=\sum_{v\in\partial U}(v)$. Then $D'$ is $U$-reduced since
$U$ is a special open set. Thus by running Dhar's algorithm for $D'$
and a point in $U$ step by step and taking the set of non-saturated
points in each step, we can get a partition of $\partial U$ by
$N'_0,N'_1,\cdots,N'_{K-1}$. Note that for every $X_i$, there exists
some $N'_k$ such that either $\partial X_i$ is a subset of $N'_k$ or
$X_i$ connects points in $\partial X_i\bigcap N'_k$ and $\partial
X_i\bigcap N'_{k+1}$, i.e., $\partial X_i\bigcap N'_k$ and $\partial
X_i\bigcap N'_{k+1}$ are nonempty and $\partial X_i\subseteq
N'_k\bigcup N'_{k+1}$. Therefore we may define a function
$\lambda:\{1,2,\cdots,n\}\rightarrow\{1,2,\cdots,K-1\}$ by
$\lambda(i)=k$ if $\partial X_i\bigcap N'_k\neq\emptyset$ and
$\partial X_i\bigcap N'_{k-1}=\emptyset$. Let
$N_k=(\bigcup_{\lambda(i)=k}B_i)\bigcup N'_k$ for
$k=0,1,\cdots,K-1$. Obviously these $N_k$'s form a partition of
$\partial U\bigcup B$. Running Dhar's algorithm for $D$ and a point
in $U$ step by step, we observe that the set of non-saturated points
in each step is precisely $N_0,N_1,\cdots,N_{K-1}$ in sequence.
Therefore the output is empty, which means $D$ is $U$-reduced.
\end{proof}

Now we come to the main conclusion of this subsection, which states
that the condition in Theorem~\ref{T:precriterion} is both
necessary and sufficient.

\begin{thm}[\textbf{Criterion for $\mathcal{L}(A)$}] \label{T:criterion}
Let $v\in \Gamma$ and let $A$ be a nonempty subset of $\Gamma$. Then
$v\in\mathcal{L}(A)$ if and only if for all special open sets $U$ containing
$v$, we have $A\bigcap U\neq \emptyset$. Furthermore,
\begin{equation*}
\mathcal{L}(A)=\bigcap_{U\in\mathcal{S}_\Gamma,U\bigcap A=\emptyset}
U^c.
\end{equation*}
In addition, $A$ is a rank-determining set if and only if all nonempty special
open sets intersect $A$.
\end{thm}
\begin{proof}
We just need to prove that if $v\in\mathcal{L}(A)$, then all critical
open sets containing $v$ must intersect $A$.

Suppose for the sake of contradiction that there exists $U\in\mathcal{S}_\Gamma$ such
that $v\in U$ and $A\bigcap U= \emptyset$. Then by
Proposition~\ref{P:2:4}, there exists a divisor $D$ such that
$\text{supp}|D|=U^c$. Thus we have $A\subseteq \text{supp}|D|$,
which means that $\mathcal{L}(A)\subseteq\text{supp}|D|$. But then
$v\notin\mathcal{L}(A)$.
\end{proof}

\begin{ex} \label{E:L}
\begin{figure}[h]
\centering
\includegraphics[width=0.75\textwidth]{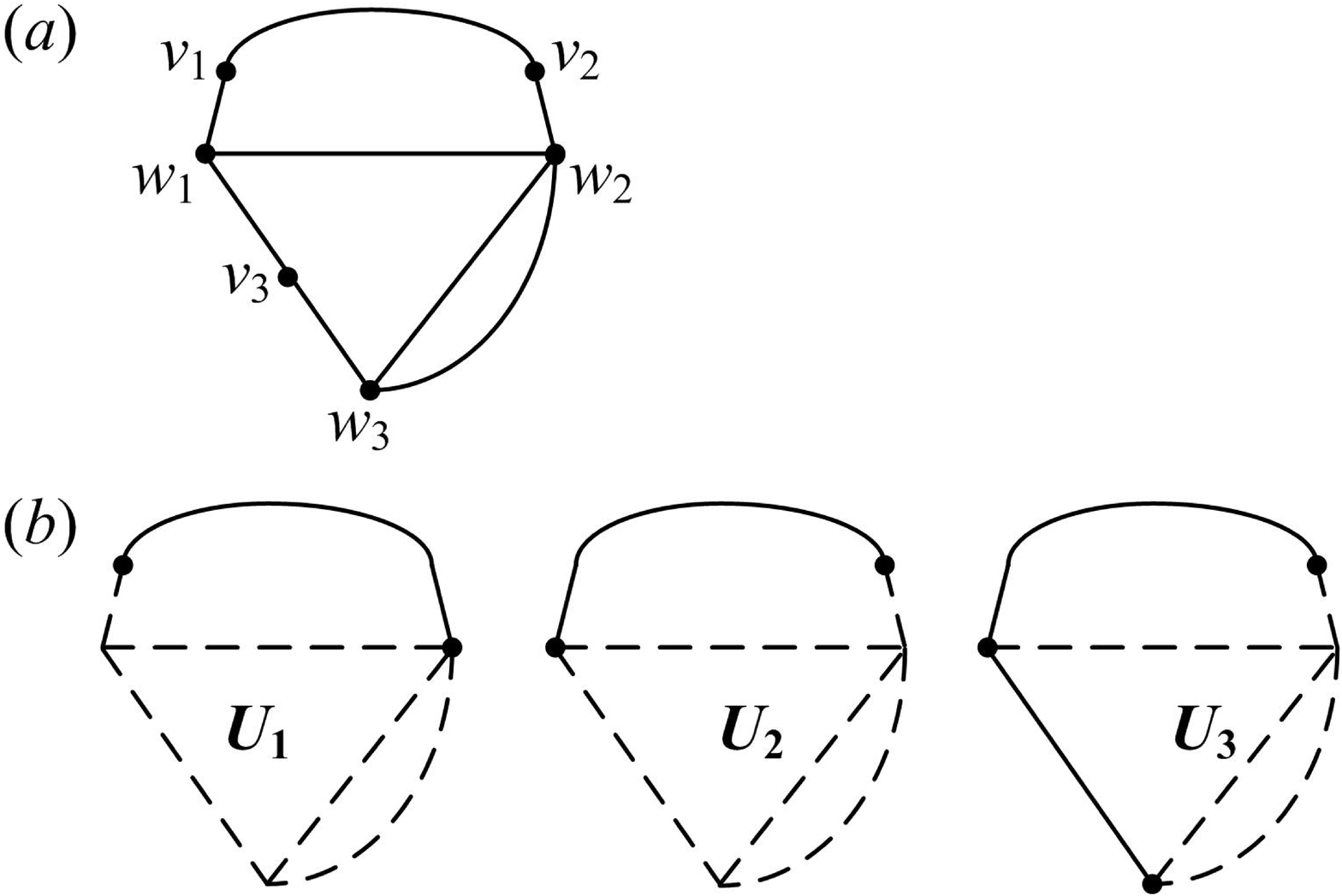}
\caption{(a) A metric graph with a vertex set $\{w_1,w_2,w_3\}$. (b)
Three examples of special open sets disjoint from
$\{v_1,v_2\}$.}\label{F:L}
\end{figure}
Let $\Gamma$ be a metric graph with a vertex set $\{w_1,w_2,w_3\}$
as shown in Figure~\ref{F:L}(a), and let $v_1,v_2,v_3$ be some internal
points. Clearly $[v_1,v_2]\subseteq\mathcal{L}(v_1,v_2)$. The dashed
areas of Figure~\ref{F:L}(b), $U_1$, $U_2$ and $U_3$, are three
examples of special open sets disjoint from $\{v_1,v_2\}$. Hence we
have $\mathcal{L}(v_1,v_2)=[v_1,v_2]$ by Theorem~\ref{T:criterion}.
Now let us consider $\mathcal{L}(v_1,v_2,v_3)$. We observe that any
special open set disjoint from $\{v_1,v_2,v_3\}$ must be a subset
of $U_3$, which implies $\mathcal{L}(v_1,v_2,v_3)=U_3^c$.
\end{ex}

\subsection{Consequences of the criterion}
\begin{cor} \label{C:2:1}
Let $A$ be a nonempty subset of $\Gamma$. If $A^c$ has $n$ connected
components $X_1,X_2,\cdots,X_n$, then $A$ is a rank-determining set
if and only if $X_i\subseteq \mathcal{L}(\partial X_i)$, for $i=1,2,\cdots,n$.
\end{cor}
\begin{proof}
For a point $v\in X_i$, if a special open set $U$ containing $v$
intersects $A$, then  $U$ must intersect $\partial X_i$. Thus by
Theorem~\ref{T:criterion}, $A$ is a rank-determining set, if and only if all
nonempty special open sets intersect $A$, if and only if for all $v\in\Gamma$,
if $v\in X_i$, then all special open sets $U$ containing $v$
intersect $\partial X_i$, if and only if $X_i\subseteq \mathcal{L}(\partial
X_i)$, for $i=1,2,\cdots,n$.
\end{proof}

\begin{cor} \label{C:2:2}
Let $\Gamma$ be a metric graph with a cut point $v$. Let $\Gamma'$
the closure of a connected component of
$\Gamma\setminus v$. Then for a nonempty subset $A$ of $\Gamma'$, we
have $\mathcal{L}_{\Gamma'}(A)\subseteq\mathcal{L}_\Gamma(A)$.
\end{cor}
\begin{proof}
For $v'\in\Gamma'$, if $v'\notin\mathcal{L}_{\Gamma}(A)$, then there
exists $U\in\mathcal{S}_\Gamma$ such that $v'\in U$ and $U\bigcap
A=\emptyset$ by Theorem~\ref{T:criterion}. Then
$U\bigcap\Gamma'\in\mathcal{S}_\Gamma'$, which means
$v'\notin\mathcal{L}_{\Gamma'}(A)$.
\end{proof}

\begin{prop} \label{P:2:5}
Let $\Gamma$ be a metric graph with a vertex set $\Omega$ and $A$ a
finite rank-determining set of $\Gamma$. Suppose there exists a
point $v$ in $A$ which has degree $m\geqslant 2$ and is not a cut
point of $\Gamma$. Let $U_v$ be an open neighborhood of $v$ such
that $(U_v\setminus v)\bigcap(\Omega\bigcup A)=\emptyset$. Denote $\Gamma-U_v$ by
$\Gamma'$. Then $\Gamma'$ is a subgraph of $\Gamma$ and
$A\setminus v$ is a rank-determining set of $\Gamma'$.
\end{prop}

\begin{proof}
$\Gamma'$ is connected since $v$ is not a cut point of $\Gamma$ and $U_v\setminus v$ contains no vertices. Thus $\Gamma'$ is a subgraph of $\Gamma$.

Clearly $U_v\setminus v$ is a disjoint union of $m$ open segments.
Denote these open segments by $e_1,e_2,\cdots,e_m$. Note that the
total number of $e_i$'s ends other than $v$ may be strictly less
than $m$ because of the existence of multiple edges.

Suppose $A\setminus v$ is not a rank-determining set of $\Gamma'$.
Then there exists $U'\in\mathcal{S}_{\Gamma'}$ disjoint from $A$ by
Theorem~\ref{T:criterion}. Without loss of generality, we assume
that $m'$ is an integer such that $e_i$ has an end in $U'$ for
$1\leqslant i\leqslant m'$ and $e_i$ has no end in $U'$ for
$m'<i\leqslant m$. Let $U=U'\bigcup(\bigcup_{i=1}^{m'}e_i)$.
Obviously $U$ is a connected open set on $\Gamma$ disjoint from $A$.
We claim $U\in\mathcal{S}_{\Gamma}$. This is because if $m'<m$, then
$(\bigcup_{i=m'+1}^{m}e_i)\bigcup v$ may glue together some of the
connected components of $\Gamma'-U'$ into one connected component of
$\Gamma-U$ while the out-degrees of those boundary points are
unchanged, and if $m'=m$, then $v$ itself forms a connected
component of $\Gamma-U$ and has out-degree at least $2$. But this
means $A$ is not a rank-determining set of $\Gamma$ by Theorem~\ref{T:criterion}, a contradiction.
\end{proof}

\begin{rmk}
The converse proposition of Proposition~\ref{P:2:5} is not true.
That is, $A$ is not guaranteed to be a rank-determining set of
$\Gamma$ by $A\setminus v$ being a rank-determining set of
$\Gamma'$. For example, let $\Gamma$ be the metric graph
corresponding to $K_4$ as shown in Figure~\ref{F:K4}. Let
$\Gamma'=[w_1,w_2]\bigcup[w_2,w_4]\bigcup [w_4,w_1]$. Then
$\{v_5,v_6\}$ is a rank-determining set of $\Gamma'$. However
$\{v_5,v_6,w_3\}$ is not a rank-determining set of $\Gamma$.
\end{rmk}

It is clear that special open sets are preserved under
homeomorphisms since out-degrees are topological invariants. Thus
Theorem~\ref{T:criterion} tells us that rank-determining sets are
also preserved under homeomorphisms (\textbf{Theorem~\ref{T:homeo}}). The
following theorem provides a more general description of this fact.

\begin{thm} \label{T:topo}
Let $f:\Gamma\rightarrow\Gamma'$ be a homeomorphism between two
metric graphs $\Gamma$ and $\Gamma'$. Let $A$ be a nonempty subset
of $\Gamma$. Then
$\mathcal{L}_{\Gamma'}(f(A))=f(\mathcal{L}_\Gamma(A))$. In
particular, $A$ is a rank-determining set of $\Gamma$ if and only if $f(A)$ is
a rank-determining set of $\Gamma'$.
\end{thm}

For a closed segment $e$ on a metric graph $\Gamma$, we say
$\phi_e:\Gamma\rightarrow\Gamma'$ is an \emph{edge contraction} of
$\Gamma$ with respect to $e$ if $\phi_e$ merges together all the
points in $e$ into a single point while mapping every point in
$\Gamma\setminus e$ to itself. Clearly an edge contraction $\phi_e$
may change the topology of $\Gamma$. We now give some some
examples which show that rank-determining sets may not be preserved
under edge contractions.

\begin{ex}
\begin{figure}[h]
\centering
\includegraphics[width=0.65\textwidth]{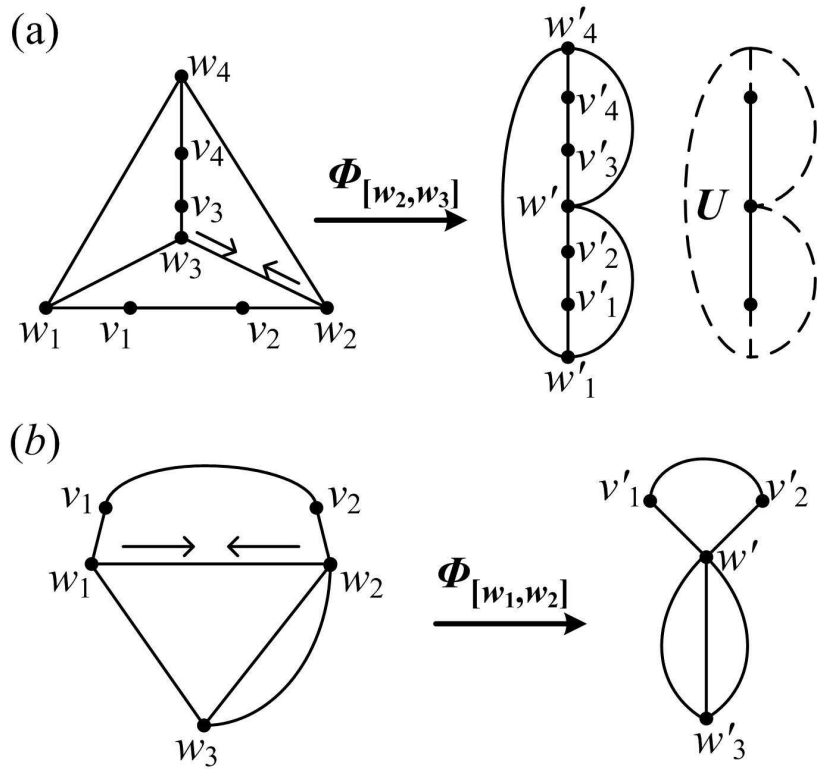}
\caption{Two examples illustrating that edge contractions do not maintain
rank-determining sets.}\label{F:Contra}
\end{figure}
(a) Consider a metric graph $\Gamma$ corresponding to $K_4$ as in
Example~\ref{E:K4}. An edge contraction with respect to $[w_2,w_3]$
results in a new graph $\Gamma'$ (Figure~\ref{F:Contra}(a)). Let
$v_1'$, $v_2'$, $v_3'$, $v_4'$, $w_1'$, $w_4'$ and $w'$ be the
points in $\Gamma'$ corresponding to $v_1$, $v_2$, $v_3$, $v_4$,
$w_1$, $w_4$ and $[w_2,w_3]$, respectively. We know that
$\{v_1,v_2,v_3,v_4\}$ is a rank-determining set of $\Gamma$.
However, as shown in Figure~\ref{F:Contra}(a), $U$ is a critical
open set disjoint from $\{v_1',v_2',v_3',v_4'\}$. Thus
$\{v_1',v_2',v_3',v_4'\}$ is not a rank-determining set of
$\Gamma'$.

(b) Now let $\Gamma$ be the metric graph as in Example~\ref{E:L}. By
contracting $[w_1,w_2]$, we get a new graph $\Gamma'$ (Figure~\ref{F:Contra}(b)). Let $v_1'$, $v_2'$, $w_3'$ and $w'$ be the
points in $\Gamma'$ corresponding to $v_1$, $v_2$, $w_3$ and
$[w_1,w_2]$, respectively. Note that
$w'\in\mathcal{L}_{\Gamma'}(v_1',v_2')$ by Corollary~\ref{C:2:2}.
Thus $\{v_1',v_2',w_3'\}$ is a rank-determining
set of $\Gamma'$. However, $\{v_1,v_2,w_3\}$ is not a
rank-determining set of $\Gamma$.
\end{ex}

\subsection{Minimal rank-determining sets}
\begin{dfn}
We say that a rank-determining set $A$ of $\Gamma$ is \emph{minimal} if
$A\setminus v$ is not a rank-determining set for every $v\in A$.
\end{dfn}

It is easy to see from Proposition~\ref{P:2:2} that minimal
rank-determining sets must be finite. In particular, the
intersection of a minimal rank-determining set and an edge contains
at most $2$ points. We have the following criterion for minimal
rank-determining sets as an immediate corollary of Theorem~\ref{T:criterion}.

\begin{prop} \label{P:mini}
Let $A$ be a subset of a metric graph $\Gamma$. Then $A$ is a
minimal rank-determining set if and only if
\begin{compactenum}[(i)]
\item all nonempty critical
open sets intersect $A$, and
\item for every point $v\in A$,
there exists a special open set that intersects $A$ only at $v$.
\end{compactenum}
\end{prop}

\begin{ex}
\begin{figure}[h]
\centering
\includegraphics[width=\textwidth]{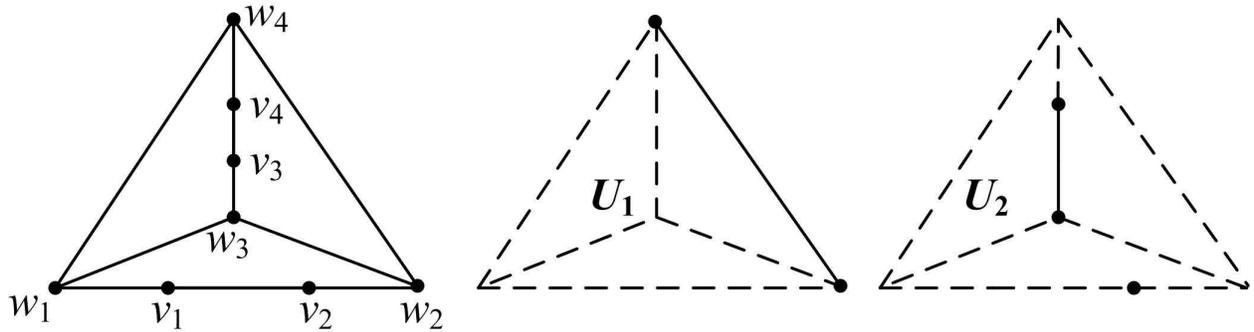}
\caption{Two examples of special open sets on the metric graph
corresponding to $K_4$.}\label{F:mini}
\end{figure}
Let us reconsider a metric graph corresponding to $K_4$ as in
Example~\ref{E:K4}. The dashed areas of Figure~\ref{F:mini}, $U_1$
and $U_2$, are two special open sets. Let $A_1=\{w_1,w_2,w_4\}$ and
$A_2=\{v_1,v_2,v_3,v_4\}$. By Example~\ref{E:K4}, $A_1$ and
$A_2$ are both rank-determining sets. We will show that they are
minimal rank-determining sets. Note that the points in
$A_{1,2}$ are symmetrically distributed. Thus we only need to find
some special open sets that intersect $A_1$ or $A_2$ at exactly one
point by Proposition~\ref{P:mini}. We observe that $U_1\bigcap A_1=\{w_1\}$ and
$U_2\bigcap A_2=\{v_1\}$. Thus $U_1$ and $U_2$ are the
desired special open sets.
\end{ex}

We've given a proof of Theorem~\ref{T:g+1} by showing constructively that a family of finite subsets of $\Gamma$, all having cardinality $g+1$, are rank-determining sets. Now we will prove that these rank-determining sets are minimal.

\begin{prop}
Let $\Gamma$ be a metric graph of genus $g$ and let $T$ be a spanning tree of $\Gamma$. Denote the $g$ disjoint open segments of $\Gamma\setminus T$ by $e_1,e_2,\cdots,e_{g}$. Choose arbitrarily a point $v_0$ from $T$ and a point $v_i$ from $e_i$ for $i=1,2,\cdots,g$. Let
$A=\{v_0,v_1,\cdots,v_g\}$. Then $A$ is a minimal rank-determining set of $\Gamma$.
\end{prop}

\begin{proof}
It suffices to find $g+1$ special open sets $U_0,U_1,\cdots,U_g$ such that $U_i\bigcap A=\{v_i\}$ for $i=0,1,\cdots,g$ by Proposition~\ref{P:mini}.

Let $U_0=\Gamma\setminus\{v_1,\cdots,v_g\}$. Clearly $U_0$ is connected and $U_0\bigcap A=\{v_0\}$. It is easy to see that $U_0$ is a desired special open set. Now let us find the remaining $g$ special open sets as required. Without loss of generality, we only need to find $U_1$ for $v_1$. Let $u_a$ and $u_b$ be the two ends of $e_1$. Note that if $x$ and $y$ are two points (not necessarily distinct) in $T$, then there exists a unique simple path (no repeated points) on $T$ connecting $x$ and $y$, which we denote $\Lambda_T^{[x,y]}$. We observe that $\Lambda_T^{[u_a,u_b]}\bigcap\Lambda_T^{[u_a,v_0]}\bigcap\Lambda_T^{[u_b,v_0]}$ contains exactly one point, which we denote $u_c$. Let $U_1=\mathcal{U}_{\{u_c,v_2,\cdots,v_g\},v_1}$. Then $U_1\bigcap A=\{v_1\}$ and a connected component of $U_1^c$ is either a single point in $\{v_2,\cdots,v_g\}$ or a closed subset $X$ of $\Gamma$ with $u_c$ on its boundary such that $\text{outdeg}_X(u_c)=2$. Thus $U_1$ is a special open set intersecting $A$ only at $v_1$. It follows that $A$ is a minimal rank-determining set of $\Gamma$.
\end{proof}

Our investigation shows that $g+1$ appears to be an upper bound for the cardinality of minimal rank-determining sets, which we formulate as a conjecture here.

\begin{conj}
Let $\Gamma$ be a metric graph of genus $g$. Then every minimal rank-determining set of $\Gamma$ has cardinality at most $g+1$.
\end{conj}


\end{document}